\documentclass[12pt,a4paper]{article}
\usepackage{amssymb,amsfonts,amsmath,amsthm,graphicx}
\usepackage[latin1]{inputenc}

\theoremstyle{plain}
\newtheorem{theorem}{Theorem}[section]

\newtheorem{remark}[theorem]{Remark}

\newtheorem{definition}[theorem]{Definition}
\newtheorem{example}[theorem]{Example}

\begin{document}
\

\pagestyle{plain}

\pagenumbering{arabic}

\begin{center}\large{\bf Partial actions of groups on quivers and path algebras }\footnote{ The second named author was supported by Fapesp, projeto tem\'atico 2014/09310-5,  by the thematic project of FAPESP  2018/23690-6,  research grants from CNPq 308706/2021-8  and  310651/2022-0. 
}\end{center}

\begin{center}
{\bf {\rm Wagner Cortes$^1$, Eduardo N. Marcos$^2$}}\end{center}

\begin{center}{ \footnotesize 1 Instituto de Matem\' atica\\
Universidade Federal do Rio Grande do Sul\\
91509-900, Porto Alegre, RS, Brazil\\
2, Departamento de Matematica\\
IME-USP, Caixa Postal 66281,\\
05315-970, S\~ao Paulo-SP, Brazil\\
E-mail: {\it wocortes@gmail.com}, {\it enmarcos@ime.usp.br}}\end{center}

\begin{abstract}

In this article, we introduce the concept of partial actions of a group $G$ on quivers and demonstrate that for any given partial action of G  on a quiver $\Gamma$, there exists another quiver, $\Gamma'$ with a full $G$-action. This is an enveloping action of the partial action of $G$ on $\Gamma$.  We also introduce partial actions of groups on algebras by subalgebras instead of ideals and we define  enveloping actions in this case.  We show that any partial action of a group on a path algebra that is induced by a partial action on a quiver has an enveloping action.

\end{abstract}

\section{Introduction}

Partial actions of groups were first introduced by R. Exel and other authors in the context of C*-algebras (see, for example, [13]). They later appeared in a purely algebraic setting (see \cite{DE}). A different perspective on partial actions emerged earlier in the work of Green and Marcos (see \cite{GM}).

Partial actions of groups have become an important tool for characterizing algebras as partial crossed products. In particular, several aspects of Galois theory can be generalized to partial group actions (see \cite{DFP}), at least under the additional assumption that the associated ideals are generated by central idempotents. Soon after the initial definition, the theory of partial actions was extended to the Hopf algebraic setting (see \cite{CJ} and \cite{ME}). \\ 

An interesting class of partial actions is the one that can be obtained by restricting a global action. This is valuable because it allows us to transfer properties from the global action to the partial one. The existence of globalization was first studied in \cite{abadie}, and numerous other results on this topic were obtained later (see, for instance, \cite{CF}, \cite{DE}, \cite{DRS}, \cite{Ferrero2}). This concept was also explored in the setting of categories by the authors in \cite{CFM}.

Path algebras are important because of the celebrated result of Gabriel in \cite{g} which states that a finite dimension algebra over an algebraically closed field is always Morita equivalent to an algebra
of the form: $kQ/I$ where $J^n \subset I \subset J^2$, where $J$ is the ideal generated by the arrow set. The quiver is then an interesting object, which is uniquely determined by the algebra and it  used in representation theory to describe modules, etc.  

In this article, we study  partial actions of groups on quivers and path algebras. We introduce the notion of partial actions of groups on quivers and prove that any such action has an enveloping action. We also define a new type of partial action of groups on algebras, focusing on subalgebras rather than the traditional approach using ideals. We apply this new notion to the study of partial actions on path algebras and show that they always have an enveloping action, consistent with our definition. Finally, we show that in our setting it is more natural to define the partial action whose partial maps  the partial isomorphisms  are  isomorphisms between subalgebras, not ideals.


\section{Partial actions}

In this section, we have our principal results. We begin with some preliminaries that are necessary to understand the paper.

Throughout this paper  $K$ always denotes a field.

The next two definitions are well known,  see \cite{ASS} for more details.

\begin{definition} 
\begin{itemize} 
\item A quiver $\Gamma$ is a quadruple 
 $\Gamma=(\Gamma_0, \Gamma_1, s, t)$, where $\Gamma_0$ is the set of vertices, $\Gamma_1$ is the set of arrows,  and $s$,$t$ are maps from $\Gamma_1$ to $\Gamma_0$ that assign a source and a target to each arrow. 
\item The path algebra of a quiver $Q$ over $K$, denoted as $K\Gamma$, is defined as follows: 

\begin{itemize} 

\item  $K\Gamma$ is a vector space over the field K. Its basis consists of all paths in $\Gamma$, including a trivial path $e_v$ of length zero for each vertex $v\in \Gamma_0$

\item The multiplication of two basis paths, p and q, is defined by concatenation: 

(i) If the source of path $p$ is the same as the target of path $q$, their product $pq$ is the new path formed by placing $p$ after $q$.

(ii) If the source of $p$ is not the same as the target of $q$, their product is zero.

(iii) Multiplication involving trivial paths is defined as: $p.e_{s(p)}=p$ and $e_{t(p)}=p$.  For any two trivial paths $e_v$ and $e_u$, their product is $e_v$ if $u=v$ and 0 otherwise. 

\end{itemize}

\end{itemize}

\end{definition}

Now, we give the following  remark that will be useful. 

\begin{remark}  We recall that a morphism of quivers $\Phi:\Gamma\rightarrow Q$ consists of a pair of maps $\Psi=(\Psi_0,\Psi_1)$, where $\Psi_0:\Gamma_0\rightarrow Q_0$ and $\Psi_1:\Gamma_1\rightarrow Q_1$ are maps that satisfy $o(\Psi_1(x))=\Psi_0(o(x))$ and $t(\Psi_1(x))=\Psi_0(t(x))$, for all $\gamma\in \Gamma_1$. We can define a composition of quiver morphisms $\Psi:\Gamma\rightarrow Q$ and $\eta:Q\rightarrow L$  by $(\eta\circ \Psi)=(\eta_0 \circ \Psi_0, \eta_1 \circ \beta_1)$. If $\Psi_0$ and $\Psi_1$ are inclusion maps then $\Psi$ is called an inclusion. Moreover, we have the identity $id_{\Gamma}$ as the map $id_{\Gamma}=(id_{\Gamma_0}, id_{\Gamma_1})$.

So, we have a category of quivers whose whose objects  are the quivers and morphims are as above. Note that if $\Psi=(\Psi_0,\Psi_1):\Gamma \rightarrow Q$ is a morphism with $\beta_0$, $\beta_1$ bijections, then we can define  $\Psi^{-1}:Q\rightarrow \Gamma$, where $\Psi^{-1}=(\Psi_0^{-1}, \Psi_1^{-1})$. Thus, we have that $\Psi\circ \Psi^{-1}=id_{Q}$ and $\Psi^{-1}\circ \Psi=id_{\Gamma}$. Hence, $\Psi$ is an isomorphism. The isomorphism between $\Gamma$ and $\Gamma$ are usually called automorphism and it is not difficult to see that the set of automorphisms of $\Gamma$ is a group and we denote it by $Aut(\Gamma)$.\end{remark}

It is convenient to point out that in  \cite{CE},  it is shown that  partial action  on a indecomposable algebra $A$ with an enveloping action is a extension by zero, a notion defined there.  

We want to define a more interesting notion of enveloping action for algebras of the type $A=K\Gamma$, where $K$ is a field. In particular, we define a "natural" notion of partial action $\alpha$ of a group $G$ on a quiver $\Gamma$. With this new notion, we have a partial action on $K\Gamma$. Moreover, if   a quiver $Q$  with a global action $\beta$ that is an enveloping action of $(\Gamma,\alpha)$, then there is a global action $\beta_1$ of a group $G$ on  $KQ$ that is an enveloping action of the induced  partial action of $G$ on  $K\Gamma$ which is denoted by  $\alpha$.

We begin with the following definition of global actions of groups on quivers.

\begin{definition} Let $G$ be a group and $\Gamma$ a quiver. A global action of $G$ on $\Gamma$ is group morphism $\beta:G\rightarrow Aut(\Gamma)$. As usual this means that we have for each $g$, an element $\beta_g\in Aut(\Gamma)$  such that  $\beta_g\circ \beta_h=\beta_{gh}$, for all $g,h\in G$.\end{definition}

Now, we introduced the concept of partial actions of groups on quivers. 

\begin{definition}  Let $\Gamma=(\Gamma_0, \Gamma_1, o, t)$ be a quiver and $G$ a group.  A partial action of $G$ on $\Gamma$ consists of a family of subquivers  $\{\Gamma^{g}, g\in G\}$ and quiver isomorphisms  $\alpha_g:\Gamma^{g^{-1}}\rightarrow \Gamma^{g}$ which satisfies the following properties: 

(i) $\Gamma^e=\Gamma$ and $\alpha_e=id_{\Gamma}$;

(ii) $(\Gamma_0, \alpha)$  is a partial action on the vertex;

(iii) $(\Gamma_1, \alpha)$ is a partial action on the arrows;

(iv) $\{t(\gamma):\gamma\in \Gamma_1^g\}\cup \{o(\gamma)|\gamma\in \Gamma_1^g\}\subseteq \Gamma_0^g$.\end{definition}

\begin{example}  Let $\Gamma$ be a quiver, $\beta$ a global of a group action of $G$ on $\Gamma$ and $\Lambda$ a subquiver of $\Gamma$. We define $\Lambda_0^g=\beta_g(\Lambda_0)\cap \Lambda_0$ and $\Lambda_1^g=\beta_g(\Lambda_1)\cap \Lambda_1$. This defines a partial action $\alpha$ of $G$ on $\Gamma$  which is called the restriction of $\beta$. \end{example}

\begin{remark}  Note that since $\alpha_g$ is an isomorphism for all $g\in G$ we have that  if $\gamma\in \Gamma^{g^{-1}}$, then  $\alpha_g(o(\gamma))=o(\alpha_g(x))$ and $\alpha_g(t(\gamma))=t(\alpha_g(x))$. In some sense, the condition (iv) is a consequence of the fact that $\alpha_g:\Gamma^{g^{-1}}\rightarrow \Gamma^{g}$ is an isomorphism of quivers for all $g\in G$. \end{remark}

Next, we introduce the definition of enveloping action of partial actions of groups on quivers.

\begin{definition} Let $\alpha$ be a partial action  of a group $G$ on a quiver $\Gamma$. A quiver $Q$ with a global action $\beta$ of $G$ on $Q$ is an enveloping action for $(\Gamma, \alpha)$ if the following conditions hold:

(a) There exists an inclusion of quivers $i:\Gamma \hookrightarrow Q$;

(b) $(\Gamma, \alpha)$ is a restriction of $(Q,\beta)$;

(c)  $(Q_0,\beta_0)$ is the enveloping action of $(\Gamma_0, \alpha_0)$;

(d) $(Q_1,\beta_1)$ is the enveloping action of $(\Gamma_1, \alpha_1)$.\end{definition}

The next theorem shows that all partial actions of groups on quivers have enveloping actions.

\begin{theorem} Let $\alpha$ be a partial action of a group $G$ on a quiver $\Gamma$. Then $(\Gamma, \alpha)$ always has an enveloping action. Moreover, two enveloping actions are isomorphic.\end{theorem}

\begin{proof} We only prove the isomorphism of two enveloping actions, since the proof of the existence is straightforward. Let $(Q_0, \beta_0)$ and $(Q_1,\beta_1)$ be the enveloping actions of $(\Gamma_0, \alpha)$ and  $(\Gamma_1, \alpha)$, respectively. We define the quiver $\Lambda=(Q_0, Q_1, o,t)$ as follows: if $b\in Q_1$, then $b=\beta_1^g(a)$, for some $a\in \Gamma_1$, $g\in G$.  We  then define $o(b)=\beta_0^g(o(a))$ and $t(b)=\beta_0^g(t(a))$. We need to show that this is well defined.  In fact, let $y\in Q_1$ such that $y=\beta_1^g(x)=\beta_1^h(z)$. Then $x=(\beta_1^{g^{-1}}\circ \beta_1^h)(z)$. Since $\beta$ is a global action, we have   $o(x)=(\beta_0^{g^{-1}}\circ \beta_0^{h})(o(z))$ and $t(x)=(\beta_0^{g^{-1}}\circ \beta_0^{h})(t(z))$. Hence, we have that  $\beta_0^g(o(x))=\beta_0^{h}(o(z))$ and $\beta_0^g(t(x))=\beta_0^{h}(t(z))$, which shows that this is well defined. We leave the remainder of this proof to the reader.\end{proof}

For our purpose in this article,  we need a different kind of partial action of groups. The  next definition  generalizes the concept of partial actions of groups in the sense of \cite{DE}. 

\begin{definition} Let $G$ be a group and $R$  a $K$-algebra. A partial action $\alpha$ of $G$ on $R$ by subalgebras consists of a family  $(R_g)_{g\in G}$ of subalgebras of $R$ and isomorphisms $\alpha:R_{g_{-1}}\rightarrow R_g$ such that the following properties hold.

(i) $R_e=R$ and $\alpha_e=id_R$;

(ii) $\alpha_g(R_{g^{-1}}\cap R_h)=R_g\cap R_{gh}$;

(iii)  $\alpha_g\circ \alpha_h(x)=\alpha_{gh}(x)$, for all $x\in \alpha_{h^{-1}}(R_h\cap R_{g^{-1}})$. \end{definition}

\begin{remark} If we require that each $R_g$ is an ideal of $R$, then  we have the usual definition  of partial actions of groups. As we will  see this requirement is too strong in the set of basic finite dimensional algebras over algebraically closed field.\end{remark}

\begin{example} Let $G$ be a group, $\beta$ a global action of $G$ on an algebra $T$ and $R$ a subalgebra of $T$. For each $g\in G$ we define $R_g=\beta_g(R)\cap R$ and $\alpha_g=\beta_g|_{R_{g^{-1}}}$. It is clear  that $R_g$ is a subalgebra of $R$ and $\alpha_g$ is an isomorphism of algebras. The  properties of the last definition are clearly satisfied. The partial action $\alpha$ obtained in this example is called the restriction of $\beta$ to $R$. \end{example}

\begin{definition} Let $\alpha$ be a partial action of $G$ on $R$ by subalgebras. An algebra $T$ with a global action $\beta$ of $G$ is a globalization for $(R,\alpha)$ if the following conditions are satisfied:

(i) There is a monomorphism of algebras $i:R\rightarrow T$ ( we view $R$ as a subalgebra of $T$).

(ii) $T$ is the subalgebra generated by $\cup_{g\in G}\beta_{g\in G}(R)$.

(iii) $D_g=\beta_g(R)\cap R$ is a subalgebra with identity of $R$.

(iv) $(R, \alpha)$ is the restriction of $(T,\beta)$. \end{definition}

\begin{remark} (i) If each $\beta_g(R)$ is an ideal of $T$, then $\sum_{g\in G}\beta_g(R)$ is the subalgebra generated by $\cup_{g\in G}\beta_g(R)$. So, our definition coincide with the usual definition in the case that all $\beta_g(R)$ are ideals.

 (ii) As we already point out let $(S,\beta)$ be an enveloping action by subalgebras of $(R,\alpha)$. If each $R_g$ is an ideal of $R$ and $R$ is an ideal of $S$, then $(S,\beta)$ is the enveloping action in the sense of [11]. 

 \end{remark}

 Let $\alpha$ be a partial action of a group $G$ on a quiver $\Gamma$. We define a partial action of $G$ on $K\Gamma$ in a natural way as follows:  for each $g\in G$ we set $(K\Gamma)_g=K\Gamma_g$ which are clearly subalgebras of $K\Gamma$ and the isomorphism $(K\Gamma)_g\rightarrow (K\Gamma)_{g^{-1}}$ induced by the quiver isomorphism $\Gamma_{g^{-1}}\rightarrow \Gamma_{g}$. We also denote this action by $\alpha$ and we call it the induced action. Note that the induced action on $K\Gamma$ is not a partial action in the sense  \cite{DE}.   
 
 The next result is our main reason to define partial actions by subalgebras.

\begin{theorem} Let $(\Gamma, \alpha)$ be a partial action on a quiver $\Gamma$ and  $(Q, \beta)$ its enveloping action. Then $(KQ, \beta)$ is the enveloping action of $(K\Gamma, \alpha)$, and any two enveloping actions are isomorphic.  \end{theorem}

\begin{proof}   Let $(Q,\beta)$ be the enveloping action of action of $(\Gamma, \alpha)$ as in Theorem 2.8.  Then  $\beta=(\beta_0, \beta_1)$,   $Q_0=\cup_{g\in G}\beta_0^g(\Gamma_0)$ and $Q_1=\cup_{g\in G}\beta_1^g(\Gamma_1)$.  Hence, $KQ$ is generated by $\cup_{g\in G}\beta_g(K\Gamma)$ (note that $Q_0\cup Q_1$ is a multiplicative basis of $KQ$). 

Now, let  $(S,\beta')$ be another of $(K\Gamma,\alpha)$. We define $\eta:KQ\rightarrow S$ by $\eta(v)=\eta(\beta_g(v))=\beta'_g(v')$ and $\eta(\alpha)=\eta(\beta_g(\alpha))=\beta'_g(\alpha)$, for some $v\in Q_0$, $a\in Q_1$, where $v=\beta_g(v')$, $a=\beta_g(a')$ with $a', v'\in K\Gamma$. Consequently, $\eta(vfu)=\eta(\beta_g(v\beta_g(f)u))=\eta(\beta_g(vfu))=\beta'_g(vfu)=\eta(v)\eta(f)\eta(u)$.  By the universal property of path algebras there is unique extension of $\eta$ to the algebra $KQ$.\end{proof}

  In item (ii) of  Definition 2.12,  we require that if $(T,\beta)$ is a globalization  of $(R,\alpha)$  then $\cup_{g\in G} \beta_g(R)$ generate as algebra the algebra $T$. The next example shows that it is not enough to assume that $\sum_{g\in G}\beta_{g\in G}\beta_g(R)=T$.

\begin{example} We start with a global action on the path algebra $KQ$. The quiver is given by  $(Q_0,Q_1,o,t)$, where $Q_0=\{1,2,3,4\}$ and $Q_1=\{a:1\rightarrow 2, b:2\rightarrow 3, c:3\rightarrow 4, d:4\rightarrow 1\}$. We consider a cyclic group $G$  of order 4 generated by $\tau$ that acts globally
 on $KQ$ as follows: $\tau(1)=2$, $\tau(2)=3$, $\tau(3)=4$, $\tau(4)=1$, $\tau(a)=b$, $\tau(b)=c$, $\tau(c)=d$ and $\tau(d)=a$.
 
 Now, let  $\Gamma=(\Gamma_0,\Gamma_1,o,t)$ where $\Gamma_0=\{2,3\}$ and $\Gamma_1=\{a,b\}$. The partial action $\alpha$ of $G$ on $\Gamma$ as follows:  $\Gamma_{\tau}^0=\{2,3\}$, $\Gamma_{\tau^2}^0=\{1,3\}$, $\Gamma_{\tau^3}^0=\{1,2\}$, $\Gamma_{\tau}^{1}=\{b\}$, $\Gamma_{\tau^{2}}^1=\emptyset$ and $\Gamma_{\tau^{3}}^1=\{a\}$. The maps are $\alpha_{\tau}^1(a)=b,$ $\alpha_{\tau^3}(b)=a$. For the vertices we have  $\alpha_{\tau}^0=\left\{\begin{array}{c}
 1\rightarrow 2\\
 2\rightarrow 3 \end{array}\right.$, $\alpha_{\tau^2}^0=\left\{\begin{array}{c}
 1\rightarrow 3\\
 3\rightarrow 1 \end{array}\right.$, $\alpha_{\tau^3}^0=\left\{\begin{array}{c}
 2\rightarrow 1\\
 3\rightarrow 1 \end{array}\right.$, $\alpha_{\tau}^1(a)=b$ and $\alpha_{\tau^3}(b)=a$.  We clearly have that the partial action  defined above is the restriction of the global action of $G$ on $Q$. Now, we easily see that there is  a partial action of $G$ on  $K\Gamma$ and a global action of $G$  on $KQ$. Note that $\sum_{g\in G}\beta_{g\in G}(K\Gamma)$ is the  vector space generated by the paths of lenght less or equal to 2, which is strictly contained in $KQ$.\end{example}

 In the next example, we show that the partial action  on the path algebra is not a partial action in the sense of Exel and Dokuchaev in \cite{DE}.

\begin{example} We consider the algebra $K\Gamma$, which is the path algebra given by the quiver $\Gamma=(\Gamma_0,\Gamma_1, o, t)$, where $\Gamma_0=\{v_1,v_2\}$ and $\Gamma_1=\{f:v_1\rightarrow v_2\}$. Let $G$ be a cyclic group of order 3 generated by $\sigma$. We define a partial action by subalgebras $\alpha$ of $G$ on $K\Gamma$ as follows: $R_e=K\Gamma$, $R_{\sigma}=Kv_1$ and $R_{\sigma^2}=Kv_2$ with the natural  algebra isomorphisms $\alpha_{e}=id$, $\alpha_{\sigma}:v_2\rightarrow v_1$ and $\alpha_{\sigma^{2}}:v_1\rightarrow v_2$.   Note that this partial action is induced by the partial action of $G$  on $\Gamma$ defined by  $\Gamma_{e}^0=\Gamma$, $\Gamma_0^{\sigma}=v_1$, $\Gamma_0^{\sigma^2}=v_2$, $\Gamma_1^{e}=\Gamma_1$, $\Gamma_1^{\sigma}=\Gamma_1^{\sigma^2}=\emptyset$. Thus,  we have the  enveloping action $(Q, \beta)$ of $\Gamma$ such that $Q=(Q_0, Q_1, o,t)$, where $Q_0=\{v_1,v_2,v_3=\sigma^{2}(v_1)\}$, $Q_1=\{f\,\, \sigma(f):v_2\rightarrow v_3,\,\,\sigma^2(f)::v_3\rightarrow v_1\}$. Consequently, we easily have that $KQ$ is the enveloping action of $K\Gamma$.

Note that the partial action on $K\Gamma$ is not a partial action in the sense of   \cite{DE}  since $Kv_i$ is not an ideal for $i\in \{1,2\}$.  \end{example}




\end{document}